\documentclass[12pt]{article}
\usepackage{amsmath,amsfonts}
\title{Commensurators and Quasi-Normal Subgroups }
\author{G. Conner and M. Mihalik}
\newtheorem{theorem}{Theorem}
\newtheorem{proposition}[theorem]{Proposition}
\newtheorem{lemma}[theorem]{Lemma}
\newtheorem{corollary}[theorem]{Corollary}

\newcounter{remarknum}
\newenvironment{remark}{\addvspace{12pt}\refstepcounter{remarknum}
\noindent{\bf Remark \arabic{remarknum}.}}{\par\addvspace{12pt}}
\newenvironment{proof}{\addvspace{12pt}\noindent{\bf Proof:}}{
$\Box$\par\addvspace{12pt}}
\newcounter{examplenum}
\newenvironment{example}{\addvspace{12pt}\refstepcounter{examplenum}
\noindent{\bf Example \arabic{examplenum}.}}{\par\addvspace{12pt}}
\newcounter{factnum}

\date{December 9, 2009}
\begin{document}
\maketitle
\begin{abstract}
If $G$ is a group, then subgroups $A$ and $B$ are {\it commensurable} if $A\cap B$ has finite index in both $A$ and $B$. The {\it commensurator} of $A$ in $G$, denoted $Comm(A,G)$, is 
$$\{g\in G | (gAg^{-1})\cap A\ \hbox {has finite index in both}\  A\ \hbox {and} \ gAg^{-1}\}.$$
It is straightforward to check that $Comm(A,G)$ is a subgroup of $G$. We say $A$ is {\it quasi-normal} in $G$ if $Comm(A,G)=G$. Denote the centralizer of $A$ in $G$ as $C(A,G)$ and the normalizer of $A$ in $G$ as $N(A,G)$ then $C(A,G)<N(A,G)<Comm(A,G)$. We develop geometric versions of commensurators in finitely generated groups. In particular, $g\in Comm(A,G)$ iff the Hausdorff distance between $A$ and $gA$ is finite. We show a quasi-normal subgroup of a group is the kernel of a certain map, and a subgroup of a finitely generated group is quasi-normal iff the natural coset graph is locally finite. This last equivalence is particularly useful for deriving asymtopitic results for finitely generated groups. Our primary goal in this paper is to develop the basic theory of quasi-normal subgroups, comparing analogous results for normal subgroups and isolating differences between quasi-normal and normal subgroups. 
\end{abstract}

\section{Introduction} 
Classically (1966), A. Borel proved several results about irreducible lattices in semisimple Lie groups that cemented commensurators as critical to the theory \cite{B}. In 1975, G. A. Margulis extended these results \cite{Ma}. 







A subgroup $Q$ of a group $G$ is {\it quasi-normal} in $G$ if $G$ is the commensurator of $H$ in $G$. Our goal in this paper is to develop the theory of quasi-normal subgroups of groups in analogy with the theory of normal subgroups of groups. There are significant parallels  between the two theories and subtle differences. We include a focus on finitely generated groups as a means to examine the geometric group theory of quasi-normal subgroups and to uncover basic geometric intuition in the subject.
 
In section 2, we derive several rather technical results. While most of these results might be skipped on a first reading, corollaries \ref{C8} and \ref{C9} give  geometric interpretations of commensurators in a finitely generated group and provide geometric motivation for what follows. In particular, we show that a subgroup $Q$ of a finitely generated group $G$ is quasi-normal iff the Hausdorff distance between $Q$ and $gQ$ is finite for every $g\in G$.  

Section 3 contains the bulk of the basic theory of quasi-normal subgroups. We show the intersection of two quasi-normal subgroups is quasi-normal, but the intersection of a countable number of quasi-normal subgroups may not be quasi-normal. The union of two quasi-normal subgroups may not generate a quasi-normal subgroup, but the union of a quasi-normal subgroup and a normal subgroup generate a quasi-normal subgroup. The ascending union of quasi-normal subgroups may not be quasi-normal.  The image and inverse image of a quasi-normal subgroup under an epimorphism is quasi-normal. We examine quasi-normal subgroups in amalgamated products and HNN extensions of groups and show that quasi-normal subgroups of a word hyperbolic group behave like normal subgroups with respect to limit sets and quasi-convexity. 

In section 4,  we produce two characterizations of quasi-normal subgroups of finitely generated groups. First we show that a subgroup of a finitely generated group $G$  is quasi-normal iff it is the kernel of a certain map of $G$ to a set. For a second characterization of quasi-normality, we consider subgroups $H$ of a finitely generated group $G$ and construct a left $H$-coset graph that is locally finite iff $H$ is quasi-normal in $G$. 
We show this graph has either 0, 1, 2 or an uncountable number of ends, in direct analogy with H. Hopf's theorem that a finitely generated group has   0, 1, 2 or an uncountable number of ends. We give an example of a finitely generated group $G$ and quasi-normal subgroup $Q$ where the number of ends of the quotient $Q\backslash\Gamma$ is countably infinite (for $\Gamma$ a Cayley graph of $G$) but the left coset graph has an uncountable number of ends. When $Q$ is a finitely generated quasi-normal subgroup of a finitely generated group $G$ then the number of ends of the coset graph agrees with the number of filtered ends of the pair $(G,Q)$ (see Chapter 14 of \cite {Ge} for a study of filtered ends of a pair of groups). Finally, we point out the connection of this second characterization to the bounded packing ideas of Hruska and Wise \cite{HW}.

In Section 5, we list some of the asymptotic results that will appear in a separate paper. These are semistability and simple connectivity at infinity results that generalize fundamental results in the subject. An analysis of Higman's simple group also appears in this section as well as a series of natural questions associated to this group. 

\section{Basic Technical Results}
 If $H$ is a subgroup of $G$ and $g\in G$, then $g$ is an element of the {\it commensurator} of $H$ in $G$ if $gHg^{-1}\cap H$ has finite index in both $gHg^{-1}$ and $H$. 
For $H$ a subgroup of $G$ we write $Comm(H,G)$ for the commensurator of $H$ in $G$ when the over group $G$ may not be apparent.

\begin{theorem}\label{T1}
If $H$ is a subgroup of $G$ and $g\in G$ then $g\in Comm(H,G)$ iff there are finite subsets $A$ and $B$ of $G$ such that for each $h\in H$ there is an $a\in A$ and  $b\in B$ such that $ha\in gH$, and $ghb\in H$. (Equivalently, there is $a\in A$ such that  $h(ag^{-1})\in gHg^{-1}$ and $b\in B$ such that $ghg^{-1}(gb)\in H$.) Corollary \ref{C8} is a geometric version of this theorem when $G$ is finitely generated.
\end{theorem}

\begin{proof}
Suppose $g\in Comm(H,G)$. Choose $h_i\in H$ such that $$\cup_{i=1}^n(gHg^{-1}\cap H)h_i=H$$ For $h\in H$, say $h=xh_i$ for some $i\in\{1,\ldots ,n\}$ and some $x\in (gHg^{-1}\cap H)$. Then $x=hh_i^{-1}  \in (gHg^{-1}\cap H)$ and $hh_i^{-1}g\in gH$. So, we can let $A$ be the finite set $\{h_1^{-1}g,\ldots , h_n^{-1}g\}$. Since $Comm(H,G)$ is a subgroup of $G$, $g^{-1}\in Comm(H,G)$. By the preceding argument, there is a finite subset $B$ of $G$ such that for each $h\in H$ there is a $b\in B$ such that $hb\in g^{-1}H$. Equivalently, $ghb\in H$. 
 
 Assume $A$ and $B$ are finite subsets of $G$ satisfying the conclusion of the theorem. Define a function $\alpha:H\to A$ such that $h\alpha (h)\in gH$. Suppose $h_1,h_2\in H$ and $\alpha (h_1)=\alpha (h_2)$.  As $h_1\alpha (h_1)$ and $h_2\alpha (h_2)$ are elements of $gH$, we have $h_1\alpha (h_1)Hg^{-1}=h_2\alpha (h_2)Hg^{-1}=gHg^{-1}$. Then 
 $$ h_2h_1^{-1}gHg^{-1}=h_2h_1^{-1}h_1\alpha (h_1)Hg^{-1}=h_2\alpha(h_2)Hg^{-1}=gHg^{-1}$$
 In particular: 
 $$\hbox{If}\  \alpha(h_1)=\alpha(h_2) \hbox{ then } h_2h_1^{-1}\in gHg^{-1}\cap H.$$
Say $im(\alpha)=\{a_1,\ldots ,a_n\}$ and select $h_i\in H$ such that $\alpha(h_i)=a_i$. If $h\in H$ and $\alpha(h)=a_i$, then $hh_i^{-1}\in H\cap gHg^{-1}$ and $h\in (H\cap gHg^{-1})h_i$. We have $H=\cup _{i=1}^n(H\cap gHg^{-1})h_i$ and $H\cap gHg^{-1}$ has finite index in $H$. 

For each $h\in H$ there is $b\in B$ such that $ghb\in H$, and so $hb\in g^{-1}H$. The preceding argument implies $H\cap g^{-1}Hg$ has finite index in $H$. Conjugating (by $g^{-1}$) gives $gHg^{-1}\cap H$ has finite index in $gHg^{-1}$. 
\end{proof}

It is now straightforward to see:

\begin{corollary}\label{subgroup}
For any group $G$ and  subgroup $H$ of $G$, $Comm(H,G)$ is a subgroup of $G$.
\end{corollary}

\begin{corollary}\label{C6}
Suppose $H$ is a subgroup of $G$ and $g\in Comm(H,G)$. There is a finite subset $A(g,H,G)$ of $HgH\subset Comm(H,G)$ and functions $\alpha_{(g,H,G)}$ and $\beta_{(g,H,G)}$ (written $A(g)$, $\alpha_g$, and $\beta_g$ respectively, when $H$ and $G$ are unambiguous) such that:

1) $\alpha_g:H\to A(g)$ and $\beta_g:H\to (A(g))^{-1}$,

2) for each $h\in H$, $h\alpha_g(h)\in gH$ and $gh\beta_g(h)\in H$,

3) $[image(\alpha_g)\cup (image (\beta_g))^{-1}]=A(g)$, and

4) for each $a\in A(g)$, there is $h_1\in H$ such that $h_1a\in gH$, and $h_2\in H$ such that $gh_2a^{-1}\in H$. (It is not possible to make $A(g)$ symmetric.)
\end{corollary}

\begin{proof}
Consider the sets $A$ and $B$ of theorem \ref{T1}. There are functions $\alpha:H\to A$ and $\beta:H\to B$ such that for each $h\in H$, $h\alpha(h)\in gH$ and $gh\beta(h)\in H$. Without loss, we assume $\alpha$ and $\beta$ are onto. Define $A(g)\equiv A\cup B^{-1}$. Define $\alpha_g:H\to A(g)$ to agree with $\alpha$ and $\beta_g:H\to (A(g))^{-1}$ to agree with $\beta$.  

If $a\in A\subset A(g)$ then any element in $\alpha^{-1}(a)$ will play the role of $h_1$ in our result. As $h_1a\in gH$, we write $h_1a=gh_2$ and $gh_2a^{-1}\in H$. If $b\in B^{-1}\subset A(g)$, let $h_2\in H$ be such that $\beta_g(h_2)=b^{-1}$. Then $gh_2\beta_g(h_2)=gh_2b^{-1}\in H$. Say $gh_2(b^{-1}=h_1$. Then $gh_2=h_1b\in gH$. 
\end{proof}

\begin{remark}
The next result applies locally. It would be interesting to extend this to a more general global result.
\end{remark}

\begin{corollary}\label{C7}
Suppose $g\in Comm(H,G)$. 

1)$A(g^{-1})$ may be selected to be $(A(g))^{-1}$ with $\alpha_{g^{-1}}\equiv\beta_g$ and $\beta_{g^{-1}}\equiv\alpha_g$. 

2) If $k\in A(g)$ then $A(k)$ may be selected to be $A(g)$ with $\alpha_k(h)\equiv\alpha_g(h_1h)$ where $h_1k=gh_1'$ for some 
 $h_1, h_1'\in H$ and $\beta_k(h)\equiv\beta_g(h_2h)$ where $gh_2=h_2'k$ for some $h_2,h_2'\in H$.
\end{corollary} 

\begin{proof}
Suppose $h\in H$ then $h\alpha_g(h)\in gH$,  and $g^{-1}h\alpha_g(h)\in H$. So, we may define $\beta_{g^{-1}}(h)\equiv \alpha_g(h)\in A(g)$.
As $gh\beta_g(h)\in H$, $h\beta_g(h)\in g^{-1}H$ and we may define $\alpha_{g^{-1}}(h)\equiv \beta_g(h)\in (A(g))^{-1}$. 

Suppose $k\in A(g)$ then $h_1k=gh_1'$ for some $h_1,h_1'\in H$. Define $\alpha_{k}:H\to A(g)$ by $\alpha_{k}(h)\equiv\alpha_g(h_1h)$. Then $h_1h\alpha_{k}(h)=h_1h\alpha_g(h_1h)\in gH=h_1kH$ and so $h\alpha_{k}(h)\in kH$, as required.  

The equality $gh_2k^{-1}=h_2'$ is valid for some $h_2, h_2'\in H$. Define $\beta_k:H\to (A(g))^{-1}$ by $\beta_k(h)=\beta_g(h_2h)$. As $gh_2h\beta_k(h)=gh_2h\beta_g(h_2h)\in H$, substituting for $g$ shows $(h_2'kh_2^{-1})h_2h\beta_k(h)\in H$ and $kh\beta_k(h)\in H$. 
\end{proof}

If $S$ is a finite generating set for a group $G$, $\Gamma(G,S)$ the Cayley graph of $G$ with respect to $S$, and $H$ a subgroup of $G$, then for any $g_1, g_2 \in G$, the {\it Hausdorff} distance between $g_1H$ and $g_2H$, denoted $D_S(g_1H,g_2H)$, is the smallest integer $K$ such that for each element  $h$ of $H$ the edge path distance from $g_1h$ to $g_2H$ in $\Gamma$ is $\leq K$ and the edge path distance from $g_2h$ to $g_1H$ in $\Gamma$ is $\leq K$. If no such $K$ exists, then $D_S(g_1H,g_2H)=\infty$. 

As a direct consequence of theorem \ref{T1} we have:

\begin{corollary} \label{C8}
Suppose $S$ is a finite generating set for a group $G$ and $H$ is a subgroup of $G$, then $g\in G$ is in $Comm(H,G)$ iff the Hausdorff distance $D_S(H,gH)<\infty$ iff $D_S(H,gHg^{-1})<\infty$. 

In particular, a subgroup $Q$ of a finitely generated group $G$ is quasi-normal in $G$ iff the Hausdorff distance $D(Q,gQ)$ is finite for all $g\in G$ iff $D(Q,gQg^{-1})$ is finite for all $g\in G$.
\end{corollary}

\begin{corollary} \label{C9}
Suppose $H$ is a subgroup of a group $G$ and $g\in Comm(H,G)$ then $gH\subset \cup_{a\in A(g)}Ha$ and $Hg\subset \cup _{a\in A(g)}aH$. (Where $A(g)$ is finite.)
\end{corollary}

\begin{proof}
By corollary \ref{C6}, for each $h\in H$ there is $a\in A(g)$ such that $gha^{-1}\in H$. Then $gh\in Ha$ and $gH\subset H\cdot A(g)$. Since $g^{-1}\in Comm(H,G)$ we have $g^{-1} H\subset H\cdot A(g^{-1})$. Inverting,  $Hg\subset (A(g^{-1}))^{-1}\cdot H=A(g)\cdot H$. 
\end{proof}

\section {Examples and Basic Facts for Quasi-Normal Subgroups}
Recall that a subgroup $Q$ of a group $G$ is quasi-normal if $G=Comm(Q,G)$. In order to check that a subgroup $Q$ of a group $G$ is quasi-normal it suffices to show that a set of generators of $Q$ is contained in $Comm(Q,G)$. This is particularly useful when $Q$ is finitely generated. 

\begin{example} \label{E1}
We show that the subgroup $\langle x\rangle$ is quasi-normal in the Baumslag-Solitar group $BS(m,n)\equiv \langle t,x:t^{-1}x^mt=x^n\rangle$.

We consider the case $G=\langle t,x:t^{-1}xt=x^2\rangle$. The other cases are analogous. Observe that $x^{-1}\langle x\rangle x=x\langle x\rangle x^{-1}=\langle x\rangle$,  $t^{-1}\langle x\rangle t\cap \langle x\rangle =\langle x^2\rangle$ and $t\langle x\rangle t^{-1} \cap \langle x\rangle=\langle x\rangle$.  $\square$
\end{example}

\begin{lemma}\label{F0} 
Suppose $Q$ is a quasi-normal subgroup of a group $G$ and $H$ is a subgroup of $G$, then $Q\cap H$ is quasi-normal in $H$. 
\end{lemma}
\begin{proof}
For each $h\in H$, $(h^{-1}Qh)\cap Q$ has finite index in both $Q$ and $h^{-1}Qh$. Then $(h^{-1}Qh)\cap Q\cap H\equiv [h^{-1}(Q\cap H)h]\cap (Q\cap H)$ has finite index in both $Q\cap H$ and $(h^{-1}Qh)\cap H\equiv h^{-1}(Q\cap H)h$. 
\end{proof}

\begin{lemma}\label{F1}
If $Q$ is a normal, finite or a finite index subgroup of a group $G$, then $Q$ is quasi-normal in $G$. If $Q$ is quasi-normal in $G$ then for any automorphism $\alpha$ of $G$, $\alpha(Q)$ is quasi-normal in $G$.
\end{lemma}

\begin{proposition}\label{F2}
Suppose $A$ and $B$ are quasi-normal subgroups of a group $G$. Then $A\cap B$ is quasi-normal in $G$. 
\end{proposition}

\begin{proof}
Fix $g\in G$ and let $\alpha_A\equiv \alpha_{(g,A,G)}$ and $\alpha_B\equiv\alpha_{(g,B,G)}$. For each $q\in A\cap B$, $q\alpha_A(q)=ga_q$ and $q\alpha_B(q)=gb_q$ for some $a_q\in A$ and $b_q\in B$. 
For each $q\in A\cap B$, define $\tau(q)\equiv a_q^{-1}b_q=\alpha_A(q)(\alpha_B(q))^{-1}$. 
Let $T$ be the finite set $\{\alpha_A(q)(\alpha_B(q))^{-1}: q\in A\cap B\}\equiv \tau(A\cap B)$. For each $t\in T$ choose $x(t)\in A\cap B$ such that $\tau(x(t))=t$. 

Now, for each $q\in A\cap B$, $\tau(q)=\tau(x(\tau(q)))$. So, $a_q^{-1}b_q=a_{x(\tau (q))}^{-1}b_{x(\tau (q))}$. 
Then $a_qa_{x(\tau (q))}^{-1}=b_qb_{x(\tau (q))}^{-1}\in A\cap B$. So for all $q\in A\cap B$, 
$$q\alpha_A(q)a_{x(\tau (q))}^{-1}=ga_qa_{x(\tau (q))}^{-1}\in g(A\cap B).$$ 
Since $\{\alpha _A(q)a_{x(t)}^{-1}: q\in A\cap B \hbox{ and } t\in T\}$ is finite, we can define $\alpha_{(g, A\cap B, G)}(q)\equiv \alpha_A(q)a_{x((\tau (q))}^{-1}$ for all $q\in A\cap B$. 

We will define $\beta_{(g, A\cap B, G)}$ in analogous fashion to complete the proof.  
Again fix $g\in G$ and let $\beta_A\equiv \beta_{(g,A,G)}$ and $\beta_B\equiv \beta_{(g, B, G)}$. For each $q\in A\cap B$, $gq\beta _A(q)=a_q'\in A$ and $gq\beta_B(q)=b_q'\in B$.

For each $q\in A\cap B$, define $\tau'(q)\equiv (a_q')^{-1}b_q'=(\beta_A(q))^{-1}\beta_B(q)$. 
Let $T'$ be the finite set $\{(\beta_A(q))^{-1}\beta_B(q): q\in A\cap B\}\equiv \tau'(A\cap B)$. For each $t\in T'$ choose $y(t)\in A\cap B$ such that $\tau'(y(t))=t$. 

For $q\in A\cap B$, $\tau'(q)=\tau'(y(\tau'(q)))$. So, $(a_q')^{-1}b'_q=(a_{y(\tau'(q))}')^{-1}b'_{y(\tau'(q))}$. Then $a_q'(a_{y(\tau'(q))}')^{-1}=b_q'(b_{y(\tau'(q))}')^{-1}\in A\cap B$. So for all $q\in A\cap B$, 
$$gq\beta_A(q)(a_{y(\tau'(q))}')^{-1}=a_q'(a_{y(\tau'(q))}')^{-1}\in A\cap B$$ 
Since $\{\beta _A(q)(a_{y(t)}')^{-1}: q\in A\cap B \hbox{ and } t\in T'\}$ is finite, we can define $\beta_{(g, A\cap B, G)}(q)\equiv \beta_A(q)(a_{y(\tau'(q))}')^{-1}$ for all $q\in A\cap B$. 
\end{proof}

The arbitrary intersection of quasi-normal subgroups need not be quasi-normal. In 1949, M. Hall Jr. proved \cite{MH} that free groups are subgroup separable.  A group $G$ is {\it subgroup separable} if any finitely generated subgroup of $G$ is the intersection of subgroups of finite index in $G$. In particular, any infinite cyclic subgroup $A$ of $F_2\equiv \langle x,y\rangle$, the free group of rank 2, is the intersection of subgroups of finite index in $F_2$.  By lemma \ref{F1}, each subgroup of finite index in $F_2$ is quasi-normal in $F_2$, but if  $A=\langle x\rangle$, then $\langle x\rangle \cap y\langle x\rangle y^{-1}=\{1\}$. So $A$ is the intersection of quasi-normal subgroups (of finite index in $F_2$), but $A$ is not quasi-normal in $F_2$.

The next example shows that the ascending union of quasi-normal subgroups is not necessarily quasi-normal. 

\begin{example} \label{E2} 
Let $$H\equiv \langle x_0,x_1, \ldots :x_{0}^{2^k}=x_{k}^2, \hbox{ for }k\geq 1, [x_i,x_j]=1\hbox{ for }i,j\geq 0\rangle\hbox{ and }$$ 
$$H_n\equiv \langle x_0,\ldots ,x_n:x_0^{2^k}=x_k^2\hbox{ for }1\leq k\leq n, [x_i,x_j]=1\hbox{ for }0\leq i,j\leq n\rangle .$$ 

The map $i_n(x_k)=x_k$ for $0\leq k\leq n$,  and the map $q_n(x_k)=x_k$ for $0\leq k\leq n$ and $q_n(x_k)=x_0^{2^{k-1}}$ for $k>n$ extend to homomorphisms $i_n:H_n\to H$ and $q_n:H\to H_n$. The composition $q_ni_n$ is the identity on $H_n$ and so the subgroup of $H$ generated by $\{x_0,\ldots ,x_n\}$ is isomorphic to $H_n$ (and a retract of $H$). We identify $H_n$ with this subgroup. As $H_0\equiv \langle x_0\rangle$ is infinite cyclic, $x_n$ is of infinite order in $H$ for all $n$. Consider the monomorphism of $H_0$ determined by $x_0\to x_0^2$. If $G$  is the resulting HNN-extenssion, then $G$ has presentation:
$$G\equiv \langle t,x_0,x_1, \ldots :t^{-k}x_{0}t^k=x_{0}^{2^k}=x_{k}^2, \hbox{ for }k\geq 1, [x_i,x_j]=1\hbox{ for }i,j\geq 0\rangle. $$ 

Note that each generator in this presentation of $G$ has infinite order.  
Now 
$$(x_i^{-1}\langle x_0\rangle x_i)\cap \langle x_0\rangle=(x_i\langle x_0\rangle x_i^{-1})\cap \langle x_0\rangle =\langle x_0\rangle$$ and
$$ (t^{-1}\langle x_0\rangle t)\cap \langle x_0\rangle=\langle x_0^2\rangle \hbox{ and } (t\langle x_0\rangle t^{-1})\cap \langle x_0\rangle=\langle x_0\rangle.$$
Hence the infinite cyclic group $\langle x_0\rangle$ is quasi-normal in $G$.

The group  $\langle x_0\rangle$ has finite index in the abelian group $H_n\equiv \langle x_0,\ldots ,x_n\rangle$. In fact $H_n/\langle x_0\rangle$ is isomorphic to $\oplus _{i=1}^n \mathbb Z_2$. By lemma \ref{F6} (below), $H_n$ is quasi-normal in $G$ for all $n\geq 0$.

The group $H$ is the ascending union of the nested groups $H_n$. We prove $H$ is not quasi-normal in $G$, by showing  $t^{-1}Ht\cap H=\langle x_0^2\rangle$ (which has infinite index in $H$). 

Suppose  $g\in t^{-1}Ht\cap H$. Let $g=t^{-1}ht\in H$ for some $h\in H$. 
By lengths of normal forms for the HNN extension $G$, it must be that $h$ is an element of the associated subgroup $\langle x_0\rangle$ (i.e. elements of the base group $H$ of the HNN extension $G$, have length 1, but  $t^{-1}ht$ has length 3 unless $h\in\langle x_0\rangle$). Now $g=t^{-1}x_0^kt=x_0^{2k}$, and so $H$ is not quasi-normal in $G$.

In lemma \ref{F5}, we show that the inverse image  of a quasi-normal subgroup under an epimorphism is quasi-normal. In our example, consider the epimorphism $q_0':G\to \langle t,x_0:t^{-1}x_0t=x_0^2\rangle$ where $q_0'(t)=t$, $q_0'(x_0)=x_0$ and $q_0'(x_k)=x_0^{2^{k-1}}$, for $k>0$. The subgroup $(q_0')^{-1} (\langle x_0\rangle)$ is quasi-normal in $G$ and has generating set $\langle x_0,x_1, tx_2t^{-1}, \ldots , t^{k-1}x_kt^{-(k-1)},\ldots \rangle$. $\square$

\end{example}

In the next example we show that the union of two quasi-normal subgroups may not generate a quasi-normal subgroup, but lemma \ref{F3} shows  the union of a quasi-normal subgroup and a normal subgroup generates a quasi-normal subgroup.  

\begin{example}\label{E3} 
Let $H$ be the group $\langle x,y:x^2=y^2\rangle$ and $G$ the HNN extension with base $H$ and associates subgroups $\langle x^2\rangle$ and $\langle x^4\rangle$. Then $G$ has presentation 
$$G\equiv\langle x,y,t:x^2=y^2, t^{-1}x^2t=x^4\rangle.$$
To see that $\langle x^2\rangle=\langle y^2\rangle$ is quasi-normal in $G$, simply observe that:
$$ x\langle x^2\rangle x^{-1}=x^{-1}\langle x^2\rangle x=y\langle x^2\rangle y^{-1}=y^{-1}\langle x^2\rangle y=\langle x^2\rangle\hbox{ and }$$
$$ t^{-1}\langle x^2\rangle t\cap \langle x^2\rangle=\langle x^4\rangle \hbox{ and } t\langle x^2\rangle t^{-1}\cap \langle x^2\rangle=\langle x^2\rangle.$$
By lemma \ref{F6}, $\langle x\rangle $ and $\langle y\rangle$ are quasi-normal in $G$. We show that $\langle x,y\rangle $ is not quasi-normal in $G$. 
Note that 
$$\langle x,y\rangle /\langle x^2\rangle \cong \mathbb Z_2\ast \mathbb Z_2.$$
It suffices to show $t^{-1}\langle x,y\rangle t\cap \langle x,y\rangle =\langle x^4\rangle$. Suppose $z\in t^{-1}\langle x,y\rangle t\cap \langle x,y\rangle$. Write $z=t^{-1}ut$ for some $u\in \langle x,y\rangle$. As $z$ is in $\langle x,y\rangle$, the base group of the HNN extension $G$, lengths of normal forms implies that $u$ is in the domain associated subgroup $\langle x^2\rangle$ (elements of the base group have length $1$, and $t^{-1}ut$ has length $3$ unless $u\in \langle x^2\rangle$). Then $z=t^{-1}x^{2k}t=x^{4k}$. $\square$
\end{example}

\begin{lemma} \label{F3} 
 If $Q$ is quasi-normal in $G$ and $N$ is a normal subgroup of $G$ then the subgroup of $G$ generated by $Q$ and $N$ is quasi-normal in $G$. Furthermore, one may arrange things so that: 
 There is a transversal $T$ for $N$ in $\langle Q\cup N\rangle$ such that $T\subset Q$, and for any $g\in G$, $t\in T$ and $n\in N$: 
 $$\alpha_{(g,\langle Q\cup N\rangle,G)}(tn)=\alpha_{(g,Q,G)}(t) \hbox{ and }\beta_{(g,\langle Q\cup N\rangle,G)}(tn)=\beta_{(g,Q,G)}(t)$$
 \end{lemma}

\begin{proof} 
In this proof, we write $A(g)$ for $A(g,Q,G)$ and $\alpha_g$ for $ \alpha_{(g,Q,G)}$ for all $g \in G$.
As $N$ is normal in $G$, each element $f\in \langle Q\cup N\rangle$ can be written as $qn$ for some $q\in Q$ and some $n\in N$. Hence there is a transversal $T\subset Q$ for $N$ in $\langle Q\cup N\rangle$.
Suppose $g\in G$,  $t\in T$ and $n\in Q$. By corollary \ref {C6} there is $q'\in Q$ such that  $t\alpha_g(t)=gq'$. 
Let $\alpha_g(t)^{-1}n\alpha_g(t)=n'\in N$.  
Then $$tn\alpha_g(t)=t\alpha_g(t) \alpha_g(t)^{-1}n\alpha_g(t)=gq'n'.$$ So we  define $\alpha_{(g,\langle Q\cup N\rangle,G)}(tn)=\alpha_{(g,Q,G)}(t)$ for all $n\in N$.

By corollary \ref {C6}, there is $\hat q\in Q$ such that  $gq\beta_g(t)=\hat q$. Let $\hat n=\beta_g(t)^{-1}n\beta_g(t)\in N$. Then 
$$gtn\beta_g(t)=gt\beta_g(t)\beta_g(t)^{-1}n\beta_g(t)=\hat q\hat n.$$
So we define $\beta_{(g,\langle Q\cup N\rangle,G)}(tn)=\beta_{(g,Q,G)}(t)$ for all $n\in N$.
\end{proof}

\begin{lemma}\label{F4}
Suppose $f:G_1\to G_2$ is an epimorphism and $Q$ is quasi-normal in $G_1$ then $f(Q)$ is quasi-normal in $G_2$ and we can arrange:

1) $A(f(g),f(Q),G_2)=f(A(g,Q,G_1))$ with 

2) $\alpha_{(f(g), f(Q), G_2)}=f(\alpha_{g})$ and 

3) $\beta_{(f(g), f(Q), G_2)}=f(\beta_{g})$ for all $g\in G_1$.
\end{lemma}

\begin{proof}
Let $g_2\in G_2$ and select $g_1\in G_1$ such that $f(g_1)=g_2$. For $q_2\in f(Q)$, let $q_1\in Q$ be such that $f(q_1)=q_2$. There is $q_1'\in Q$ such that $q_1\alpha_{g_1}(q_1)=g_1q_1'$. Then $q_2f(\alpha_{g_1}(q_1))=g_2f(q_1')$. So we may select $\alpha_{(g_2, f(Q), G_2)}=f(\alpha_{g_1})$ (as long as $f(g_1)=g_2$). Similarly $\beta_{(g_2, f(Q), G_2)}=f(\beta_{g_1})$.
\end{proof}

\begin{lemma} \label{F5}
Suppose $f:G_1\to G_2$ is a homomorphism and $Q$ is quasi-normal in $G_2$ then $f^{-1}(Q)$ is quasi-normal in $G_1$. Furthermore, as $Q\cap f(G_1)$ is quasi-normal in $f(G_1)$ (see lemma \ref {F0}) we can arrange:

1) $f(A(g_1,f^{-1}(Q), G_1))=A(f(g_1),Q\cap f(G_1),f(G_1))$,

2) $f(\alpha_{(g_1,f^{-1}(Q),G_1)}(q'))= \alpha _{(f(g_1),f(G_1)\cap Q,f(G_1))}(f(q'))$ for all $q'\in f^{-1}(Q)$, and 

3) $f(\beta_{(g_1,f^{-1}(Q),G_1)}(q'))=\beta _{(f(g_1),f(G_1)\cap Q,f(G_1))}(f(q'))$ for all $q'\in f^{-1}(Q)$.

\noindent Note that when $f$ is an epimorphism, 1), 2) and 3) simplify to:

$1'$) $f(A(g_1,f^{-1}(Q), G_1))=A(f(g_1),Q,(G_2)$,

$2'$) $f(\alpha_{(g_1,f^{-1}(Q),G_1)}(q'))= \alpha _{(f(g_1), Q,G_2)}(f(q'))$ for all $q'\in f^{-1}(Q)$, and 

$3'$) $f(\beta_{(g_1,f^{-1}(Q),G_1)}(q'))=\beta _{(f(g_1),Q,G_2)}(f(q'))$ for all $q'\in f^{-1}(Q)$.

\end{lemma}

\begin{proof}
Lemma \ref{F0} implies $Q\cap f(G_1)$ is quasi-normal in $f(G_1)$ and $f^{-1} (Q)=f^{-1}(Q\cap f(G_1))$. Hence we may assume that $f$ is an epimorphism.

Let $T$ be a transversal for $N\equiv ker (f)$ in $f^{-1}(Q)$. Fix $g_1\in G_1$. For each element $a$ of $A(f(g_1),Q,G_2)$ pick $a'\in f^{-1}(a)$.  For $t\in T$ and $n\in N$, then $f(t)\in Q$, $\alpha _{(f(g_1),Q,G_2)}(f(t))\in A(f(g_1),Q,G_2),\hbox{  and }$
$$f(tn(\alpha _{(f(g_1),Q,G_2)}(f(t)))')=f(t)\alpha _{(f(g_1),Q,G_2)}(f(t))=f(g_1)q'\hbox{ for some }q'\in Q.$$
This implies $tn(\alpha _{(f(g_1),Q,G_2)}(f(t)))'\in g_1(f^{-1}(Q))$. 
Hence for all $t\in T$ and $n\in N$, we define $\alpha_{(g_1,f^{-1}(Q),G_1)}(tn)$ to be  $(\alpha _{(f(g_1),Q,G_2)}(f(t)))'$. Similarly, it makes sense to define $\beta_{(g_1,f^{-1}(Q),G_1)}(tn)$ to be  $(\beta _{(f(g_1),Q,G_2)}(f(t)))'$.
\end{proof}

\begin{lemma} \label{F6}
Suppose $Q$ is a quasi-normal subgroup of $G$ and $Q'$ is a subgroup of $G$ such that either $Q'$ has finite index in $Q$, or $Q$ has finite index in $Q'$, then  $Q'$ is quasi-normal in $G$. 
\end{lemma}

\begin{proof}
For each $g\in G$, consider the finite set $A(g,Q,G)$ and functions $\alpha_{(g,Q,G)}:Q\to A(g,Q,G)$ and $\beta_{(g,Q,G)}:Q\to A(g,Q,G)^{-1}$. If $Q'$ has finite index in $Q$ choose cosets $Q'q_1,\ldots , Q'q_n$ covering $Q$. Then for $g\in G$, and $q'\in Q'$, $q'\alpha_{(g,Q,G)}(q')=gq$ for some $q\in Q$. Pick $i\in \{1,\ldots , n\}$ such that $q=\bar qq_i$ for some $\bar q\in Q'$. Then $q'(\alpha_{(g,Q,G)}(q')q_i^{-1})=g\bar q$ and we may set $\alpha_{(g,Q',G)}(q')=\alpha_{(g,Q,G)}(q')q_i^{-1}$.

For $q'\in Q'$, there is $\bar q\in Q$ such that $gq'\beta_{(g,Q,G)}(q')=\bar q$. There is $\bar q'\in Q'$ and  $j\in \{1,\ldots , n\}$ such that $\bar q=\bar q'q_j$. So, $gq'\beta_{(g,Q,G)}(q')q_j^{-1}=\bar q'\in Q'$
Hence, we may choose $\beta_{(g,Q',G)}(q')=\beta_{(g,Q,G)}(q') q_j^{-1}$ and $Q'$ is quasi-normal in $G$.

Next assume $Q$ has finite index in $Q'$ and suppose the cosets $Qq'_1,\ldots , Qq'_n$ cover $Q'$. Let $q'\in Q'$. Then $q'=qq_i$ for some $q\in Q$ and some $i\in \{1,\ldots , n\}$. Now $q\alpha_{(g, Q,G)}(q)=g\bar q$ for some $\bar q\in Q\subset Q'$. So $ q'=g\bar q(\alpha_{(g, N,G)}(q))^{-1}q_i$. 
Hence we may define $\alpha_{(g,Q',G)}(q')=q_i^{-1}\alpha_{(g,Q,G)}(q)$.  

For $q'\in Q'$ there is $q\in Q$ and $i\in \{1,\ldots ,n\}$ such that $q'=qq_i$. Then $gq\beta_{(g,Q,G)}(q)=\bar q\in Q\subset Q'$ and $gq'q_i^{-1}\beta_{(g,N,Q)}(q)=\bar q$ and we may let $\beta_{g, Q',G)}=q_i^{-1}\beta_{(g,Q,G)}(q)$.
\end{proof}

\begin{lemma} \label{F7}
If $f:H\to H$ is a monomorphism and $f(H)$ has finite index in $H$, then $H$ (and by the previous fact $f(H)$) is a quasi-normal subgroup of the (ascending) HNN extension $G=H\ast _f$ 
\end{lemma}

\begin{proof}
Consider the presentation of $H\ast_f$ given by: $$\langle t,H:t^{-1}ht=f(h) \hbox{ for all } h\in H\rangle.$$ 
Let $f(A)h_1,\ldots ,f(A)h_m$ be $f(H)$ cosets covering $H$. 
It suffices to show that $\{t,t^{-1}\}\subset Comm(H)$.

For each $h\in H$, $t^{-1}ht=f(h)\in H$. Hence $ht=tf(h)\in tH$ for all $h\in H$. 
For each $h\in H$, $th=tf(h')h_i$ for some $h'\in H$ and some $i\in \{1,\ldots ,m\}$. Then $th=t(t^{-1}h'th_i)=h'(th_i)$, and $th(th_i)^{-1}=h'\in H$. By theorem \ref{T1},  $t\in Comm(H)$. 

To see that $t^{-1}\in Comm(H)$, first observe that for each $h\in H$, $t^{-1}ht=f(t)\in H$. 
Next, observe that for each $h\in H$, $h=t^{-1}h'th_i$ 
for some $h'\in H$ and some $i\in \{1,\ldots ,m\}$. Then $h(th_i)^{-1}=t^{-1}h'\in t^{-1}H$. By theorem \ref{T1}, $t^{-1}\in Comm(H)$.
\end{proof}

\begin{lemma}\label{amal}
If $G=G_1\ast_QG_2$ and $Q$ is quasi-normal in $G_i$ for $i\in \{1,2\}$ then $Q$ is quasi-normal in $G$.
\end{lemma}
\begin{proof} 
For each $g\in \{G_1\cup G_2\}$, $g^{-1}Qg\cap Q$ has finite index in both $Q$ and $g^{-1}Qg$, by hypothesis. As $G_1\cup G_2$ generates $G$ and is a subset of the subgroup $Comm(Q,G)$ of $G$ (see corollary \ref{subgroup}), $G=Comm(Q,G)$.
\end{proof}

 In \cite{MTo}, M. Mihalik and W. Towle proved that an infinite quasi-convex subgroup of a word hyperbolic group has finite index in its normalizer. 
 The same proof shows:
\begin{theorem} \label{quasi} 
Suppose $H$ is an infinite quasi-convex subgroup of a word hyperbolic group $G$ then $H$ has finite index in its commensurator. 
\end{theorem}
\begin{proof} 
(Outline) Let $a$ be an element of $Q\equiv Comm(H,G)$. It suffices to bound the distance from $a$ to $H$ in $\Gamma$, a Cayley graph of $G$ with respect to some finite generating set (containing a set of generators for $H$). Let $\alpha\equiv (\ldots , h_{-1} h_0, h_1, \ldots )$ be a bi-infinite geodesic in the generators of $H$ (so that $\alpha$ is quasi-geodesic in $\Gamma$). Assume the initial vertex of $h_0$ is $a\equiv x_0$ and the initial point of $h_n$ is $x_n$. Choose $N$ large, with respect to the Hausdorff distance $D\equiv D(H,aH)$ in $\Gamma$. Let $x$ be a point of  $H$ within $D$ of $x_N$. Consider the geodesic rectangle $([1,a], [a,x_N], [x_N,x], [x,1])$. By thin geodesic triangles, some $x_i$ (for $1\leq i\leq N$) is within $D_1 $ of $x_i'\in H$ (where $D_1$ only depends on $\delta $, the thin triangle constant and the quasi-convexity constants for $\alpha$ and $H$). Similarly there is a $j$ such that $-N\leq j\leq -1$ such that $x_j$ is within $D_1$ of $x_j'\in H$. 

The geodesic quadrilateral $([x_j,x_i],[x_i,x_i'],[x_i',x_j'],[x_j',x_j])$ has (opposite) sides of length $\leq D_1$, implying each point of $[x_j,x_i]$ is close to each point of $[x_j',x_i']$. As $a$ is close to $[x_j,x_i]$, $a$ is close to $[x_j',x_i']$ and so $a$ is close to a point of $H$. 
\end{proof}

\begin{theorem}\label{limit} 
The limit set of an infinite quasi-normal subgroup of a word hyperbolic group $H$ is the entire boundary of $H$. 
\end{theorem}
\begin{proof}
(This proof is basically the same as the standard one for normal subgroups.) Let $Q$ be an infinite quasi-normal subgroup of a word hyperbolic group $H$. Let $\Gamma$ be a Cayley graph for $H$ (on a finite generating set). 

$(\ast)$ By the definition of quasi-normal, the limit set of $\partial Q=\partial (hQ)$ in $\Gamma$ for all $h\in H$. 

As word hyperbolic groups have only finitely many conjugacy classes of finite subgroups, $Q$ contains an element $a$, of infinite order. Let $a^{\pm \infty}=\partial \langle a\rangle$ in $\Gamma$. Let $\alpha=(\ldots , a_{-1},a_0,a_1,\ldots )$ be a bi-infinite geodesic edge path in $\Gamma$ with $\partial (\alpha)=a^{\pm \infty}$. As elements of infinite order determine quasi-geodesics in $\Gamma$, $\alpha$ is of bounded distance $D$ from $Q\subset \Gamma$. In particular, $a^{\pm\infty}\in \partial Q$.

Let $x_0$ be the initial vertex of the edge $a_0$, $\beta=(b_1,b_2,\ldots )$ a geodesic edge path in $\Gamma$ beginning at $x_0$,  $b^{\infty}$ the boundary point of $\beta$, $y_i$ the initial point of $b_i$, and let $h_i\in H$ be the group element such that $g_ix_0=b_i$. Consider the ideal triangle with sides $g_i\alpha$, $[x_0,g_i(a^\infty))$ and $[x_0,g_i(a^{-\infty}))$. Since $b_i$ is a vertex of $g(\alpha)$, one of the two sides of the ideal triangle, $[x_0,g_i(a^\infty))$ or $[x_0,g_i(a^{-\infty}))$, 
passes within $\delta$ (the hyperbolic constant for thin triangles in 
$\Gamma$) of $b_i$. Hence $b^{\infty}$ is a limit point of the boundary points of 
$\{g_i(\alpha)\}_{i=1}^\infty$. As $g_i(\alpha)$ is within $D$ of $g_i(Q)$,  $\partial (g_i(\alpha))\subset \partial (g_iQ)=\partial Q$ (see $(\ast)$). Hence $b\in \partial Q$. As $\beta$ was arbitrary $\partial Q=\partial H$. 
\end{proof}

\section{Characterizations of quasi-normal subgroups of finitely generated groups}

In this section we produce two characterizations of quasi-normal subgroups of finitely generated groups that connect the theory to both well developed and emerging ideas in group theory.

\begin{lemma} \label{estimate}
Suppose $Q$ is a subgroup of the finitely generated group $G$.  Fix a finite generating set, $S$, for $G$, and let $|\cdot|$ be the corresponding word-length norm on $G$, let $d$ be the induced left invariant word metric on $G$ where~$d(a,b)=|b^{-1}a|$, and $D$ be the corresponding Hausdorff metric on subsets of $G$.\\

\noindent Suppose $Q$ is quasi-normal in $G$. Let $k = \underset{s \in S}{\max}(D(s Q,Q))+1.$  Then for all $a,b \in G$ we have the following:
\begin{enumerate}
\item \label{left} $D(bQ,Q) \leq k|b|$
\item \label{conj} $D(bQb^{-1},Q) \leq (k+1)|b|$
\item $D(QbQ,Q) \leq k|b|$.
\item $D(aQbQ,abQ)=D(QbQ,bQ) \leq 2k|b|$\\

\end{enumerate}
\end{lemma}
\begin{proof}
Let $b=b_1b_2\cdots b_n$ where each $b_i \in S$.  Then 
$$D(bQ,Q) \leq D(bQ,b_1\cdots b_{n-1}Q)+D(b_1\cdots b_{n-1}Q,b_1\cdots b_{n-2}Q)+\cdots+D(b_1Q,Q)$$
$$=D(b_nQ,Q)+D(b_{n-1}Q,Q)+\cdots +D(b_1Q,Q)\leq k |b|$$ 
by left invariance.  Then 
$$D(bQb^{-1},Q)\leq D(bQb^{-1},bQ) + D(bQ,Q) \leq |b|+ k|b|=(k+1)|b|.$$

Also, for any $q,q' \in Q$, 
$$d(q'bq,Q) = d(bq,Q) \leq D(bQ,Q) \leq k|b|\hbox{ and }$$
$$d(QbQ,q)=d(QbQ,1) \leq d(QbQ,b) + d(b,1)=0+|b|\leq k|b|.$$

\noindent Next, 
$$D(aQbQ,abQ)=D(QbQ,bQ) \leq D(QbQ,Q) +D(Q,bQ) \leq k|b| + k|b| = 2k|b|.$$
\end{proof}

If $f:G\to A$ is a function from a group to a set then define the {\it invariant set} of $f$ to be:Ê
$$I(f)=\{x \in G: f(gx)=f(g) \hbox{ for all } g \in G\}.$$ Ê
If $A$ is also a group, and $f$ is a homomorphism, then the invariant set of $f$ is $ker(f)$. 
It is easy to check that the invariant set for any function is a group:

If $x, y \in I(f)$ then $f(gyx)=f(gy)=f(g)$ (so $xy\in I(f)$), and $f(gx^{-1})=f(gx^{-1}x)=f(g)$ for all $g \in G$ (so $x^{-1}\in I(f)$).  

A function $\phi :G\to L$  from a group to a set is defined to be a {\it quasi-homomorphism} if there is an integer $k$ such that for all $a,b\in G$, 
$$D(\phi^{-1}(\phi(a))\cdot \phi^{-1}(\phi(b)), \phi^{-1}(\phi(ab)))\leq 2k|b|,$$ 
Define $\ker(\phi) \equiv \phi^{-1}(\phi(1_G))$.

\begin{theorem}\label{ker} 
A subset $Q$ of a finitely generated group $G$ is a quasi-normal subgroup of $G$ iff there is a set $L$ and $Q$ is the \emph{kernel} of a quasi-homomorphism $\phi :G\to L$ iff $Q$ is the invariant set of $\phi$. 
\end{theorem}

\begin{proof}
Fix a generating set, $S$, for $G$, use the notation of lemma \ref{estimate},
and suppose $Q$ is a quasi-normal subgroup of $G$ and let $L$ be the set of left cosets $\{gQ \mid g \in G\}$
of $Q$ in $G$.  Let $\phi$ be the natural map from $G$ to $L$ taking $g \in G$ to the left coset $gQ$ (so $Q=\ker(\phi)$). Let $k = \underset{s \in S}{\max}(D(s Q,Q))+1$ as in lemma \ref{estimate}. Then $D(\phi^{-1}(\phi(a))\cdot \phi^{-1}(\phi(b)), \phi^{-1}(\phi(ab))) = d(aQbQ, abQ)\leq 2k|b|$ again by lemma \ref{estimate}.\\

Conversely, suppose  $\phi:G\to L$ is a quasi-homomorphism.
We proceed to show that $\ker(\phi)$ is the invariant set of $\phi$ and a quasi-normal subgroup of $G$.  For convenience, let  $Q\equiv\ker(\phi)$.  Let $g \in G$, then 
$$D(\phi^{-1}(\phi(g))\cdot \phi^{-1}(\phi(1_G)), \phi^{-1}(\phi(g\cdot1_G)))\leq 2k|1_G|=0\hbox{ implying}$$ 

$$(\ast) \ \ \ \phi^{-1}(\phi(g))\cdot Q=\phi^{-1}(\phi(g))\hbox{ for all }g \in G$$



If $q\in Q$ then by $(\ast)$, $gq\in \phi^{-1}(\phi(g))$ for all $g\in G$. Then $\phi(gq)=\phi(g)$ for all $g\in G$, and $Q\subset I(\phi)$.
 
If $y\in I(\phi)$, then $\phi(yx)=\phi(x)$ for all $x \in G$. In particular for $x=1_G$ we have $\phi(y)=\phi(1_G)$ and $y\in Q$. Thus $Q$ is equal to the invariant set of $f$. In particular, $Q$ is a subgroup of $G$.  

Finally, by $(\ast)$ we see that each set  $\phi^{-1}(\phi(a))$ is a union of left cosets of $Q$ and since it contains $a$ it contains $aQ$.

By hypothesis we have 
$$D(\phi^{-1}(\phi(a^{-1}))\cdot \phi^{-1}(\phi(a)),Q)\leq 2k|a|\hbox{ for all }a \in G,\hbox{ and so }$$ 
$$\sup_{q,q'\in Q} d(aqa^{-1}q',Q) \leq 2k|a|.$$ 
For $q'=1_G$ we obtain  $\sup_{q\in Q} d(aqa^{-1},Q) \leq 2k|a|,$ and  
$$\sup_{q\in Q} d(aq,Q) \leq (2k+1)|a|.$$ 
Conversely, left invariance yields 
$$\sup_{q\in Q} d(q,a^{-1}Q) \leq (2k+1)|a^{-1}|\hbox{ for all }a^{-1} \in G.$$  Thus $D(aQ,Q) \leq (2k+1)|a|$ for all  $a \in G$.
\end{proof}

Our second characterization of quasi-normality is based on the following lemma. A converse to lemma \ref{locfin} will be proved in theorem \ref{locfinequiv}.

 \begin{lemma} \label{locfin} 
Suppose $Q$ is a quasi-normal subgroup of a finitely generated group $G$. Let $S$ be a finite generating set for $G$. Then are only finitely many cosets $qsQ$ where $q\in Q$ and $s\in S^{\pm1}$. (Equivalently, there are only finitely many  cosets $gQ$ such that, in the Cayley graph $\Gamma(G,S)$, an edge connects a vertex of $Q$ to a vertex of $gQ$.)
\end{lemma}
\begin{proof}
Suppose $g_1Q,g_2Q,\ldots $ are distinct cosets in $G$, and $e_i$ is an edge of $\Gamma$ (say with label $t_i$) that begins in $Q$ and ends in  $g_iQ$. Infinitely many of the $t_i$ must be identical and so we may assume all $t_i$ are the same label, call it $t$. Say $q_i\in Q$ is the initial point of $e_i$. Consider the (left) action of $q_jq_i^{-1}\in Q$ on $\Gamma$. This element fixes the set $Q$, takes the edge $e_i$ to $e_j$ and takes the coset $g_iQ$ to the coset $g_jQ$. Hence, if the Hausdorff distance between $Q$ and $g_iQ$ is $K$, then the Hausdorff distance from $Q$ to $g_jQ$ is $K$ for all $j$. For each $j>0$, let $\alpha_j$ be an edge path of length $\leq K$ in $\Gamma$ from $1$ (the identity vertex) to a point of $g_jQ$ . There are only finitely many edge paths at $1$ in $\Gamma$ with length $\leq K$ and so two of these paths must agree. But then two of $g_iQ$ cosets  intersect non-trivially, contrary to our hypothesis. 
\end{proof}

Suppose $G$ is a group with finite generating set $S$ and $H$ is a subgroup of $G$. Let $\Lambda(S,H,G)$ be the graph with vertices the left cosets $gH$ of $G$ and a directed edge (labeled $s$) from $gH$ to $fH$ if for some $s\in S$ and $h_1, h_2\in H$, we have $gh_1sh_2=f$. (Equivalently, in the Cayley graph $\Gamma(S,G)$, there is an edge labeled $s$ with initial point in $gH$ and end point in $fH$.)
The following result is a direct consequence of lemma \ref{locfin}.
\begin{corollary}\label{locfin2} 
Suppose $G$ is a group with finite generating set $S$ and $Q$ is quasi-normal in $G$. Then the graph $\Lambda(S,Q,G)$ is locally finite and $G$ acts  (on the left) transitively on the vertices of $\Lambda(S,Q,G)$ and by isometries (using the edge path metric)  on $\Lambda(S,Q,G)$. The stabilizer of $gQ$ is $gQg^{-1}$ and the quotient map $p:\Gamma(S,G)\to \Lambda(S,Q,G)$ commutes with the left action of $G$. 

\end{corollary}

\begin{remark}
Observe that if $Q$ is a quasi-normal subgroup of $G$, $S$ is a finite generating set for $G$ and $v$ is a vertex of $\Lambda(S,Q,G)$ then there may be many edges emanating from $v$ with label $s\in S$. 
\end{remark}

H. Hopf \cite {H} and H. Freudenthal \cite {F} developed the theory of the number of ends of a finitely generated group. In particular if $G$ is a group with finite generating set $S$, then the Cayley graph $\Gamma(G,S)$ has either 0, 1, 2 or an uncountable number of ends. R. Geoghegan's book \cite {Ge} gives a complete analysis of the proper homotopy theory of ends of groups and it is our standard reference for this subject.

For $n\in\{1,2,\ldots \}$,  a connected locally finite CW-complex has $n$-{\it ends} if there is a compact set $C\subset X$ such that $X-C$ has $n$ unbounded components, and there is no compact set $D\subset X$ such that $X-D$ has more than $n$ unbounded components. We say $X$ has an {\it infinite number of ends} if for any $n\in \{1,2,\ldots\}$ there is a compact $C\subset X$ such that $X-C$ has at least $n$ unbounded components. 
A continuous function $f:X\to Y$ is {\it proper} if for each compact set $C\subset Y$, $f^{-1}(C)$ is compact in $X$.  
For a connected CW-complex $X$, the {\it set of ends} of $X$ is the set of equivalence classes $[r]$ where $r:[0,\infty)\to X$ is an (infinite) proper edge path in $X$. Proper edge paths $r$ and $s$ are equivalent (or {\it converge to the same end})  if for any compact set $C$ of $X$ there is an edge path in $X-C$ from the image of $r$ to the image of $s$. It is straightforward to show that the number of ends of $X$ agrees with the cardinality of the set of ends of $X$. If $S$ is a finite generating set for the group $G$ and $N$ is a normal subgroup of $G$, then the number of ends of the group $G/N$ is the same as the number of ends of $\Gamma(S,G)/N$ and hence is 0, 1, 2 or uncountable. This need not be the case for a quasi-normal subgroup $Q$ of a finitely generated group $G$ (see example \ref {E4}). Instead, the graph $\Lambda(S,Q,G)$ seems a more appropriate  object of analysis than $Q\backslash\Gamma(S,G)$. This line of reasoning is verified in our paper \cite {CM} where the graph $\Lambda(S,Q,G)$ is fundamental in developing the results of that paper.  
 
\begin{theorem}\label{Qends}
If $G$ is a group with finite generating set $S$, and $Q$ is a quasi-normal subgroup of $G$ then $\Lambda(S,Q,G)$ has $0$, $1$, $2$ or an uncountable number of ends, and this number is independent of the finite generating set $S$. 
\end{theorem}
\begin{proof}
As $G$ acts transitively on the vertices of $(S,Q,G)$  and by isometries on $\Lambda(S,Q,G)$ the standard proof that a Cayley graph of a group has $0$, $1$, $2$ or an uncountable number of ends can be modified to show that $\Lambda(S,Q,G)$ has $0$, $1$, $2$ or an uncountable number of ends. I.e. if $\Lambda(S,Q,G)$ has at least 3 ends, let $K$ be a finite subgraph of $\Lambda(S,Q,G)$ such that $\Lambda(S,Q,G)-K$ has $n\geq 3$ unbounded components. Choose one of the unbounded components $A$ of $\Lambda(S,Q,G)-K$ and an element $g\in G$ so that  $gK\subset A$ and $gK$ is far from $K$. Then $\Lambda(S,Q,G)-(K\cup gK)$ has at least $2(n-1)$ unbounded components. A standard argument continuing this line of reasoning shows $\Lambda(S,Q,G)$ has an uncountable number of ends.  

If $T$ is another finite generating set for $G$, then the graph $\Lambda(T,Q,G)$ has the same set of vertices as does $\Lambda(S,Q,G)$ (the left cosets $gQ$). Let $f_S:\Lambda(S,Q,G)\to \Lambda(T,Q,G)$ be defined as follows: $f_S$ restricted to the vertices of $\Lambda(S,Q,G)$ is the identity. Suppose $s\in S$. Choose a $T$-word $w_s$ such that in $G$, $s=w_s$. If $e$ is any directed edge of $\Lambda(S,Q,G)$ with label $s\in S$ and initial vertex $g_1Q$ and terminal vertex $g_2Q$, let $\tilde e$ be an edge of $\Gamma(S,G)$ with label $s$, initial vertex $v_1\in g_1Q$ and terminal vertex $v_2\in g_2Q$. The edge path $ \tau_{\tilde e}$ at $v_1$ with labeling defined by $w_s$ ends at $v_2$. Define $f_S$ to linearly map $e$ to the edge path $\tau_e\equiv p(\tau_{\tilde e})$ of $\Lambda(T,Q,G)$ (where $p:\Gamma(S,G)\to \Lambda(S,Q,G)$ is the quotient map). Note that $\tau_e$ is an edge path from $g_1Q$ to $g_2Q$.  Both $f_T$ and $f_S$ are proper and the compositions $f_Sf_T$ and $f_Tf_S$ are the identity on vertices so $f_T$ and $f_S$ induce isomorphisms between the set of ends of $\Lambda(S,Q,G)$ and $\Lambda(T,Q,G)$.
\end{proof}

\begin{remark}
Theorem \ref{Qends} suggests that for $G$ a group with finite generating set $S$, and $Q$ a quasi-normal subgroup of $G$ the number of ends of $\Lambda(S,Q,G)$ is an appropriate choice for the {\it number of ends of the pair $(G,H)$}. The standard definition of the number of ends of the pair $(G,Q)$ is the number of ends of $Q\backslash \Gamma(G,Q)$.  Chapter 14 of Geoghegan's book \cite{Ge}, presents a comparison between the standard number of ends of a pair of groups and the number of filtered ends of a pair of group. In a separate paper we show that the number of ends of $\Lambda(S,Q,G)$ is the same as the number of filtered ends of the pair $(G,Q)$ when $Q$ is finitely generated. 
\end{remark}   

Suppose $S$ is a finite generating set for the group $G$ and $Q$ is a quasi-normal subgroup of $G$. Corollary \ref{C9} suggests the graphs  $\Lambda(S,Q,G)$ and $Q\backslash\Gamma(S,G)$ are quasi-isometric. This is not the case as determined by the following example.

\begin{example}\label{E4} 
If $G=\langle t,x:t^{-1}xt=x^2\rangle$ and $Q=\langle x\rangle$, then $\Lambda(\{x,t\},Q,G)$ is  a tri-valent tree. The graph of $Q\backslash\Gamma(S,G)$ is obtained as follows: Begin with a ray, with vertices labeled $v_i$ for $i\leq 0$. Assume the directed edge from $v_{i-1}$ to $v_i$ is labeled $t$. There is a loop labeled $x$ at each $v_i$. Call this graph $A_0$. Attach a directed edge labeled $t$ to $v_0$ with end vertex $v_1$ and a loop of length 2 to $v_1$ with each edge labeled $x$. Let $w_1$ label the vertex of this loop opposite $v_1$. Call this graph $\hat B_1$. Let $\hat B_1' $ be another copy of $\hat B_1$ and attach $\hat B_1$ to $\hat B_1'$ along the respective loops of length 2 with a half twist (so that $v_1$ is identified with $w_1'$ and $w_1$ is identified with $v_1'$). Call the resulting graph $A_1$. Note that $A_1$ has 2-ends. 
 
Next, attach an edge at $v_1$ labeled $t$ with end vertex $v_2$ and attach a loop of length $4$ to $v_2$ such that each edge of the loop is labeled $x$. Let $w_2$ label the vertex of this loop opposite $v_2$. Attach to this graph an edge labeled $t$ beginning  at  $w_1$ and ending at $w_2$.  Call the resulting graph $\hat B_2$.  Let $\hat B_2' $ be another copy of $\hat B_2$ and attach $\hat B_2$ to $\hat B_2'$ along the respective loops of length 4 with a one quarter twist. Call the resulting graph $A_2$. Note that $A_2$ has 4-ends. 

Next, attach an edge at $v_2$ labeled $t$ with end vertex $v_3$ and attach a loop of length $8$ to $v_2$ such that each edge of the loop is labeled $x$. Attach to this graph three additional edges, each labeled $t$ and each beginning  at a vertex of the loop at $v_2$  and ending at a vertex at the loop at $v_3$ so that the relations $t^{-1}xt=x^2$ is satisfied.  Call the resulting graph $\hat B_3$.  Let $\hat B_3' $ be another copy of $\hat B_3$ and attach $\hat B_3$ to $\hat B_3'$ along the respective loops of length 8 with a one eighth twist. Call the resulting graph $A_3$. Note that $A_3$ has 8-ends. 
Continue to construct $Q\backslash\Gamma(S,G)$. 

The number of ends of $Q\backslash\Gamma(S,G)$ is countable, while the number of ends of $\Lambda(\{x,t\},Q,G)$ is uncountable. As the cardinality  of the set of ends of a graph is a quasi-isometry invariant, the graphs $\Lambda(\{x,t\},Q,G)$ and $Q\backslash\Gamma(S,G)$ are not quasi-isometric. $\square$
\end{example}

In \cite{HW},  C. Hruska and D. Wise make the following definition:

\noindent {\bf (Bounded packing)} Let $G$ be a discrete group with a left invariant
metric $d$. Suppose also that $d$ is proper in the sense that every metric ball is finite. A
subgroup $H$ has {\it bounded packing} in $G$ (with respect to $d$) if, for each constant $D$,
there is a number $N=N(G,H,D)$ so that for any collection of $N$ distinct cosets
$gH$ in $G$, at least two are separated by a distance of at least $D$.
(Here $d(g_1H, g_2H)$ is the infimum of $d(g_1h_1, g_2h_2)$ for all $h_1,h_2\in H$.)

The main theorem of \cite {HW} is the following bounded packing result (which is more general and more sophisticated than theorem \ref{quasi}):

\medskip

\noindent {\bf Theorem (Hruska-Wise)} Let $H$ be a relatively quasiconvex subgroup of a relatively hyperbolic
group $G$. Suppose $H \cap gPg^{-1}$ has bounded packing in $gPg^{-1}$ for each conjugate of
each peripheral subgroup $P$. Then $H$ has bounded packing in $G$.

\medskip

Consider the graph  $\Lambda(S,Q,G)$ where $S$ is a finite generating set for the group $G$ and $Q$ is quasi-normal in $G$. For $\Gamma(S,G)$ the Cayley graph of $G$ with respect to $S$, the natural projection map $P:\Gamma\to C$ respects the left action of $G$ on $\Gamma$ and $C$. If $d$ and $D$ are the edge path metrics on $\Gamma$ and $C$ respectively, and $d(g_1Q,g_2Q)$ is the infimum of $d(g_1q_1,g_2q_2)$ for all $q_1,q_2\in Q$, then $D(g_1Q,g_2Q)\leq d(g_1Q,g_2Q)$ for all $g_1,g_2\in G$. As $G$ acts transitively on $C$,  and since $C$ is locally finite (corollary \ref{locfin2}) we have:

\begin{lemma} \label{bdpk} 
If $G$ is finitely generated and $Q$ is quasi-normal in $G$, then $Q$ has bounded packing in $G$.
\end{lemma}

For  $F(x,y)$, the free group on $\{x,y\}$, it is elementary to show that $\langle x\rangle $ has bounded packing in $F(x,y)$. Certainly $\langle x\rangle$ is not quasi-normal in $F(x,y)$. 

 If the wording is slightly changed in the bounded packing definition for finitely generated groups, then one gets quasi-normal.

\begin{lemma}\label{bdpk2}
Let $G$ be a finitely generated group with word
metric $d$. A
subgroup $H$ is quasi-normal in $G$ if, for each constant $D$,
there is a number $N=N(G,H,D)$ so that for any collection of $N$ distinct cosets
$gH$ in $G$, at least one is separated from $H$ by a distance of at least $D$.
(Here again,  $d(H, gH)$ is the infimum of $d(h_1, gh_2)$ for all $h_1,h_2\in H$.)
\end{lemma}
\begin{proof}
Let $g\in G$ and say $d(1,g)=D$. By hypothesis, there are only finitely many distinct  cosets $g_1H,\ldots, g_NH$ of distance $\leq D$ from $H$. For each $h\in H$, $hgH$ is within $D$ of $H$ and so $hgH=g_iH$ for some $i$. Without loss we assume $g_1=g$. Let $H_i=\{h\in H : hg\in g_iH\}$ so that $\{H_i\}_{i=1} ^n$ partitions $H$ (reindexing, we may assume that $H_i\ne \emptyset$). Note that  $h\in H_1$ iff $hg\in gH$ iff $g^{-1}hg\in H$, so $H_1=H\cap gHg^{-1}$. 

Suppose $h_i\in H_i$. It is straightforward to show that $ h_i^{-1}H_i\subset H_1$ and $h_iH_1\subset H_i$. Hence $H_i=h_iH_1$. Now $H=\cup _{i=1}^Nh_iH_1$ and $H_1 $ has finite index in $H$. This implies that there is an integer $D_g$ such that each point of $H$ is within $D_g$ of $H_1$. As each point of $H_1$ is within $|g|=D$ of $gH$, each point of $H$ is within $D+D_g$ of $gH$. Now, each point of $H$ is within $D+D_{g^{-1}}$ of $g^{-1}H$, so each point of $gH$ is with in $D+D_{g^{-1}}$ of $H$, and $H$ is quasi-normal in $G$.  
\end{proof}

\begin{theorem}\label{locfinequiv}
Suppose $G$ is a group with finite generating set $S$. Then the subgroup $Q$ of $G$ is quasi-normal in $G$ iff the left coset graph $\Lambda(S,Q,G)$ is locally finite.\end{theorem}
\begin{proof} 
Lemma \ref{locfin2} implies $C$ is locally finite when $Q$ is quasi-normal in $G$. As $G$ acts transitively on $C$ in any setting, $C$ is not locally finite iff the valence of each vertex is infinite, but then lemma \ref{bdpk2} implies $Q$ is not quasi-normal in $G$. (Note that if $S$ contains $n$ elements, then there at most $2n$ labeled (and directed) edges connecting two given vertices of $C$. In particular, if $C$ is not locally finite, each vertex has infinitely many adjacent vertices.)
\end{proof}

\section{Closing Remarks}

In a separate paper \cite{CM}, we show that certain asymptotic aspects of quasi-normal subgroups are in direct analogy with those for normal subgroups. In particular, we prove the following results. 

\begin{theorem}\label{SS} 
Suppose $G$ is a finitely generated group, and $Q$ is an infinite, finitely generated, quasi-normal subgroup of $G$ of infinite index in $G$, then $G$ is one-ended and semistable at infinity.
\end{theorem}
The corresponding result for normal subgroups is the main result of Mihalik's paper \cite{M}.
As a straightforward corollary to theorem \ref{SS}, we obtain Vee Ming Lew's theorem on semi-stability of groups with infinite finitely generated subnormal subgroups of infinite index. As a direct corollary to theorem \ref{SS} and lemma \ref{F7}, we obtain:

\begin{theorem} \label{HNN} 
Suppose $H$ is a finitely generated group, $\phi:H\to H$ is a monomorphism and $\phi(H)$ has finite index in $H$. If $G$ is the resulting HNN extension:
$$G\equiv \langle t, H: t^{-1}ht=\phi(h) \hbox{ for all }h\in H\rangle$$
then $G$ is semistable at infinity.
\end{theorem}

\begin{theorem}\label{SC}
Suppose $G$ is a finitely presented group, and $Q$ is a subgroup of $G$ that is infinite, one ended, finitely presented, quasi-normal in $G$, and of infinite index in $G$, then $G$ is simply connected at infinity.
\end{theorem}

Theorem \ref{SC} generalizes B. Jackson's corresponding result for normal subgroups in \cite{J}.

\begin{remark} Higman's group 
$$G\equiv \langle a_1,\ldots ,a_4: a_i^{-1}a_{i+1}a_i=a_{i+1}^2 \hbox{ cyclically for all } i\rangle$$ 
is an infinite finitely presented group with no non-trivial subgroups of finite index. A proper normal subgroup $N$ of $G$ is {\it maximal} if it is not contained in any other proper normal subgroup. As the ascending union of normal subgroups is normal, any proper normal subgroup of $G$ is contained in maximal proper normal subgroup of $G$. (To see this,  list the elements of $G$ as $g_1, g_2, \ldots$. If $N$ is a proper normal subgroup of $G$, let $N_0=N$ and $N_i$ be the normal closure of $N_{i-1}\cup \{g_i\}$ if this group is not $G$ and otherwise let $N_i=N_{i-1}$. Now $\{N_i\}_{i=0}^\infty$ is an ascending sequence of normal subgroups in $G$. Hence $M\equiv \cup _{i=0}^\infty N_i$ is normal in $G$. The group $M$ is a proper normal subgroup of $G$ since otherwise, the generators, $a_i$ are elements of $M$ for all $i$, and so for some $j$, $a_i$ is in $N_j$ for all $i$ (this is impossible since $N_j\ne N$). Now $M$ is maximal since if $g\in G$ is not in $M$ then the normal closure of $M\cup \{g\}$ is $G$. As $G$ has no subgroups of finite index, $G/M$ is an infinite finitely generated simple group. 

a) The ascending union of quasi-normal subgroups is not necessarily a quasi-normal subgroup. (See example \ref{E2})

b) Does Higman's group contain interesting quasi-normal subgroups?

c) Is $G/M$ quasi-simple (i.e. does it contain non-trivial quasi-normal subgroups)? 
\end{remark}



\end{document}